\documentclass[a4paper, 12pt]{amsart}
\usepackage{amssymb,amsmath}
\makeatletter
\@namedef{subjclassname@2010}{%
  \textup{2010} Mathematics Subject Classification}
\makeatother

\newcommand{\C}{\mathbb{C}}

\newcommand{\Z}{\mathbb{Z}}

\newcommand{\Q}{\mathbb{Q}}

\newcommand{\Spec}{\operatorname{Spec}}

\newcommand{\Ker}{\operatorname{Ker}}

\newcommand{\dg}{{\rm div}}

\newcommand{\ld}{{\rm LD}}

\newcommand{\trdeg}{{\rm tr.deg}\:}
\newtheorem{thm}{Theorem}[section]
\newtheorem{prop}[thm]{Proposition}
\newtheorem{lem}[thm]{Lemma}

\newtheorem{example}[thm]{Example}

\begin{document}
\title[Closed polynomials and their applications]{Closed polynomials and their applications for computations of kernels of monomial derivations}
\author{Chiaki Kitazawa}
\address[C. Kitazawa]{Graduate School of Science and Technology, Niigata University, 8050 Ikarashininocho, Nishi-ku, Niigata 950-2181, Japan}
\email{f16a050k@alumni.niigata-u.ac.jp}
\author{Hideo Kojima}
\address[H. Kojima]{Department of Mathematics, Faculty of Science, Niigata University, 8050 Ikarashininocho, Nishi-ku, Niigata 950-2181, Japan}
\email{kojima@math.sc.niigata-u.ac.jp}
\author{Takanori Nagamine}
\address[T. Nagamine]{Graduate School of Science and Technology, Niigata University, 8050 Ikarashininocho, Nishi-ku, Niigata 950-2181, Japan}
\email{t.nagamine14@m.sc.niigata-u.ac.jp}
\date{July 24, 2018}
\subjclass[2010]{Primary 13N15; Secondary 13A50, 13B25.}
\keywords{Closed polynomial; Derivation; Darboux polynomial.}
\thanks{Research of the second author was partially supported by Grant-in-Aid for Scientific Research (C) (No.\ 17K05198) from Japan Society for the Promotion of Science}
\thanks{Research of the third author was partially supported by Grant-in-Aid for JSPS Fellows  (No.\ 18J10420) from Japan Society for the Promotion of Science}

\begin{abstract}
In this paper, we give some results on closed polynomials and factorially closed polynomial in $n$ variables which are generalizations of results in \cite{KN15}, \cite{N16} and \cite{N18}. In particular, we give a characterization of factorially closed polynomials in $n$ variables over an algebraically closed field for any characteristic. 
Furthermore, as an application of results on closed polynomials, we determine kernels of non-zero monomial derivations on the polynomial ring in two variables over a UFD. Finally, by using this result and the argument in \cite[\S 5]{NZ06}, for a field $k$, we determine the non-zero monomial derivations $D$ on $k[x,y]$ such that the quotient field of the kernel of $D$ is not equal to the kernel of $D$ in $k(x,y)$.  
\end{abstract}
\maketitle

\setcounter{section}{-1}

\section{Introduction} 
Let $k[{\rm X}]$ be the polynomial ring in $n$ variables over a field $k$ of characteristic zero and let $k({\rm X})$ be its quotient field. For a $k$-derivation $D$ on $k[{\rm X}]$, we denote its kernel by $k[{\rm X}]^D$. The $k$-derivation $D$ is naturally extended to a $k$-derivation on $k({\rm X})$, which is denoted by the same notation $D$, and its kernel is denoted by $k({\rm X})^D$. The study of derivations on polynomial rings and their kernels have been motivated in various areas of mathematics. Kernels of $k$-derivations in $k[{\rm X}]$ and $k({\rm X})$ have been studied by many mathematicians. See, e.g., \cite{N94}, \cite{PM190}, \cite{EMS136} for recent excellent accounts.
It is well-known that the kernel of any $k$-derivation on $k[{\rm X}]$ with $n \leq 3$ is finitely generated as a $k$-algebra and that the kernel of any non-zero $k$-derivation on $k[{\rm X}]$ with $n = 2$ can be expressed as $k[f]$ for some $f \in k[{\rm X}]^D$, which are originally given in \cite{NN88}. Note also, the fact holds true in the case where $k$ is a UFD of characteristic zero and $n = 2$ (see \cite[Corollary 3.2]{EK04}). However, it is difficult to determine the generator of $k[{\rm X}]^D$ of some $k$-derivation $D$ on $k[{\rm X}]$ even if $k[{\rm X}]^D$ is finitely generated as a $k$-algebra.  

On the other hand, the second and third authors studied closed polynomials in \cite{KN15}, \cite{N16} and \cite{N18}. Here, a non-constant polynomial $f \in k[{\bf X}] \setminus k$ is a \emph{closed polynomial} if the ring $k[f]$ is integrally closed in $k[{\rm X}]$. Of cause, closed polynomials are define by the same way in the case where $k$ is an integral domain (see Section 1). It is well known that the kernel of a derivation $D$ on $k[{\bf X}]$ is integrally closed in $k[{\bf X}]$. In particular, if $\trdeg_kk[{\bf X}]^D = 1$, then it is generated by a closed polynomial. Thus, closed polynomials may be useful to determine the generator of $k[{\rm X}]^D$. 

In section 1, we recall some kinds of derivations and some concepts; coordinates, closed polynomials and factorially closed polynomials over an integral domain. 
In section 2, we give some results on closed polynomials and factorially closed polynomials which are generalizations of results in \cite{KN15}, \cite{N16} and \cite{N18}. Moreover, in Example \ref{ex:2-7}, we show that Danielewski surface and Koras-Russell threefold are factorially closed polynomials, but they are not coordinates. 
In section 3, as an application of results on closed polynomials in section 2, we study kernels of monomial derivations on the polynomial ring in two variables over a UFD. This is in Theorem \ref{thm:3-3}. Also, in Lemma \ref{lem:3-1}, we give a way to find polynomials which are vanished by a given derivation. 
Finally, in section 4, by using the argument in \cite[\S 5]{NZ06} and Theorem \ref{thm:3-3}, we determine the non-zero monomial derivations $D$ on $k[x,y]$ such that the quotient field of the kernel of $D$ is not equal to the kernel of $D$ in $k(x,y)$. 

\section{Preliminaries}  
Let $R$ be an integral domain and let $R[{\rm X}] = R[x_1, \ldots, x_n]$ be the polynomial ring in $n$ variables $x_1, \ldots, x_n$ over $R$. We denote $Q(R)$ by the quotient field of $R$. For ${\bf w} = (w_1, \ldots, w_n) \in (\Z_{\geq 0})^n$, we define a degree function on $R[{\rm X}]$ by the map 
	\begin{center}
	  $\deg_{\bf w} : R[{\rm X}] \to \Z_{\geq 0} \cup \{ -\infty \}$
	\end{center}
defined by $\deg_{\bf w} x_i = w_i$ for $1 \leq i \leq n$ and $\deg_{\bf w} 0 = -\infty$. For $(1, \ldots, 1) \in (\Z_{\geq 0})^n$, we denote simply $\deg f := \deg_{(1, \ldots, 1)} f $. If $\deg_{\bf w} f \geq 2$, 
then we denote by $\ld_{\bf w}(f)$ the smallest positive prime number dividing $\deg_{\bf w} f$. For example, if $\deg_{\bf w} f$ is a prime number, then $\ld_{\bf w}(f)$ = $\deg_{\bf w} f$.
A non-constant polynomial $f \in R[{\rm X}] \setminus R$ is a \emph{closed polynomial} if the ring $R[f]$ is integrally closed in $R[{\rm X}]$. For a polynomial $f \in R[{\rm X}]$, we define 
	\begin{center}
	  $\hat{f} := {\rm gcd}(f_{x_1}, \ldots, f_{x_n})$, 
	\end{center}
where $f_{x_i}$ is the partial derivative of $f$ with respect to $x_i$ and we take the greatest common divisor of $f_{x_1}, \ldots, f_{x_n}$ as a polynomial in $Q(R) \otimes_RR[{\rm X}]$.  
A polynomial $f \in R[{\rm X}]$ is called a \emph{coordinate} if there exist polynomials $f_2, ..., f_n \in R[{\rm X}]$ such that $k[f, f_2, ..., f_n] = R[{\rm X}]$. 

Let $B$ be an $R$-algebra. For an $R$-derivation $D$ on $B$, we say that $D$ is \emph{irreducible} if the only principal ideal of $B$ containing the image of $D$ is $B$ itself. We say that $D$ is \emph{locally nilpotent} if for any $f \in B$, there exists $m \geq 0$ such that $D^m(f) = 0$. We denote also the kernel of an $R$ derivation by $B^D$. An $R$-subalgebra $A \subset B$ is \emph{factorially closed} in $B$ if for all $f, g \in B \setminus \{ 0 \}$, $fg \in A$ implies that $f$ and $g$ belong to $A$. If $R[f]$ is factorially closed in $R[{\rm X}]$, then we call $f \in R[{\rm X}]$ a \emph{factorially closed polynomial}. We can see easily that the following holds true. 
\begin{lem} \label{lem:1-1}
{\rm (cf.\ \cite[Proposition 2.4]{N18})}
Let $f \in R[{\rm X}] \setminus R$. Then the following assertions hold true. 
	\begin{enumerate}
	  \item
	  If $f$ is a coordinate, then it is a factorially closed polynomial. 
	  \item
	  If $f$ is a factorially closed polynomial, then it is a closed polynomial. 
	\end{enumerate}
\end{lem} 
\noindent 
Later, in Examples \ref{ex:2-3} and \ref{ex:2-7}, we give some examples of such polynomials. 

A polynomial $f \in R[{\bf X}]$ is called an \emph{integral element} or \emph{a Darboux polynomial} for an $R$-derivation $D$ if $D(f) \in fR[{\bf X}]$, that is, $D(f) = hf$ for some $h \in R[{\bf X}]$. We define an abelian monoid $X_{D}$ by 
	\begin{center}
 	  $X_{D} = \{ h \in R[{\bf X}] \ | \ D(f) = hf$ for some $f \in R[{\bf X}] \setminus \{ 0 \} \}$. 
	\end{center}

We often use the following result to verify whether a given polynomial is closed or not.  
\begin{thm} \label{thm:N16} 
{\rm (cf.\ \cite[Proposition 3.11]{N16})}
Let $R$ be an integral domain of characteristic zero and let $R[{\rm X}] = R[x_1, \ldots, x_n]$ be the polynomial ring in $n$ variables over $R$. Let $f \in R[{\rm X}] \setminus R$ such that $Q(R)[f] \cap R[{\rm X}] = R[f]$. Assume that there exists an element ${\bf w} \in (\Z_{\geq 0})^n$ satisfying one of the following two conditions{\rm :} 
	\begin{enumerate}
	  \item[{\rm (a)}]
	  $\deg_{\bf w} f = 1$, 
	  \item[{\rm (b)}]
	  $\deg_{\bf w} f \geq 2$ and $\displaystyle \deg_{\bf w} \hat{f} < \frac{\ld_{\bf w}(f) - 1}{\ld_{\bf w}(f)} \deg_{\bf w} f$.
	\end{enumerate}
Then $f$ is a closed polynomial.
\end{thm}
For a non-constant polynomial $f \in R[{\bf X}] \setminus R$, we can verify whether the condition ``$Q(R)[f] \cap R[{\rm X}] = R[f]$" is satisfied or not by using the following lemma. 
\begin{lem} \label{lem:1-3} 
Let $R$ be a UFD. For a non-constant polynomial $f \in R[{\rm X}] \setminus R$, we denote $c(f) \in R$ by the greatest common divisor of the coefficients of $f$. Then the following two conditions are equivalent to each other{\rm :} 
	\begin{enumerate}
	  \item
	  $c(f - f(0, \ldots, 0)) \in R^*$. 
	  \item
	  $Q(R)[f] \cap R[{\rm X}] = R[f]$. 
	\end{enumerate}   
\end{lem}
\begin{proof}
Let $K := Q(R)$ and $B := R[{\rm X}]$. Without loss of generality, we may assume that $f(0, \ldots, 0) = 0$. Then $c(f - f(0, \ldots, 0)) = c(f)$. We note also $c(gh) = c(g)c(h)$ for $g, h \in B$. 

{\rm \bf (1) $\Rightarrow$ (2)} 
Suppose that $c(f) \in R^*$. Let $g \in K[f] \cap B$. Then there exist $u_0, u_1, ..., u_m \in K$ such that
	\begin{center}
	  $g = u_0f^m + u_1f^{m - 1} + \cdots + u_{m - 1}f + u_m$. 
	\end{center}   
Since $f(0, \ldots, 0) = 0$ and $g \in B$, we see that $g(0, \ldots, 0) = u_m$ and $u_m \in R$. Now, we choose $r \in R \setminus \{ 0 \}$ with $ru_i \in R$ for $0 \leq i \leq m$. Let $g_1 := r(g - u_m)/f \in B$, namely, $g_1 = \sum_{i = 0}^{m - 1}ru_{m - 1 - i}f^i$. Then $c(g_1) = rc(f)^{-1}c(g - u_m) \in rR$. Hence $g_1 \in rB$, especially, $ru_{m - 1} = g_1(0, ..., 0) \in rR$. This implies $u_{m - 1} \in R$. Next, let $g_2 := r(g - u_{m - 1}f - u_m)/f^2 \in B$. By the same augment, we have $u_{m - 2} \in R$. Using the same augment inductively, we have $u_i \in R$ for $0 \leq i \leq m$, so $g \in R[f]$.   

{\rm \bf (2) $\Rightarrow$ (1)} 
Suppose that $c(f) \notin R^*$. Let $f^* := f/c(f) \in B$. Then $R[f] \subsetneqq R[f^*]$ and $K[f] = K[f^*]$. Since $c(f^*) \in R^*$, by the consequence of {\rm ``(1) $\Rightarrow$ (2)"}, $K[f^*] \cap B = R[f^*]$. Therefore   
	\begin{center}
	  $R[f] \subsetneqq R[f^*] = K[f^*] \cap B = K[f] \cap B = R[f]$. 
	\end{center}
This is a contradiction. 
\end{proof} 

\section{Closed polynomials and related topics}
In this section, we study closed polynomials and related topics. Some results in this section are generalizations for some results of papers written by the second and third authors \cite{KN15}, \cite{N16} and \cite{N18}. Let $R$ be an integral domain and let $R[{\rm X}] = R[x_1, \ldots, x_n]$ be the polynomial ring in $n$ variables $x_1, \ldots, x_n$ over $R$. 
\begin{example} \label{ex:2-1} {\rm (cf.\ \cite[Example 4.2]{N16})}
Let $R$ be an integral domain and $ux_1^{m_1}\cdots x_n^{m_n}$ be a monomial of $R[{\rm X}] = R[x_1, \ldots, x_n]$. Then the following two conditions are equivalent to each other{\rm :}  
	\begin{enumerate}
	  \item
	  $ux_1^{m_1}\cdots x_n^{m_n}$ is a closed polynomial. 
	  \item
	  $u \in R^*$ and ${\rm gcd}(m_1, \ldots, m_n) = 1$. 
	\end{enumerate}
\end{example}
\begin{proof}
{\bf (1) $\Rightarrow$ (2)}  
Let $f = ux_1^{m_1}\cdots x_n^{m_n}$. If $u \notin R^*$, then $x_1^{m_1}\cdots x_n^{m_n} \notin R[f]$, but it is integral over $R[f]$. Thus, $R[f]$ is not integrally closed in $R[{\rm X}]$. Now, we suppose that $u \in R^*$ and $d := {\rm gcd}(m_1, \ldots, m_n) \geq 2$. For $1 \leq i \leq n$, let $l_i = m_i/d \in \Z_{\geq 0}$. Then $u^{-1}f = (x_1^{l_1}\cdots x_n^{l_n})^d$, so $x_1^{l_1}\cdots x_n^{l_n} \notin R[f]$, but it is integral over $R[f]$.  

{\bf (2) $\Rightarrow$ (1)} 
Let $f = ux_1^{m_1}\cdots x_n^{m_n}$ and let ${\bf w} = (1, \ldots, 1)$. Then $f$ is ${\bf w}$-homogeneous. Since $u \in R^*$, we have $Q(R)[f] \cap R[{\rm X}] = R[f]$. Also, ${\rm gcd}(m_1, \ldots, m_n) = 1$ means that $f$ is primitive in $Q(R) \otimes_RR[{\rm X}]$, that is, there are no ${\bf w}$-homogeneous polynomials $g \in Q(R) \otimes_RR[{\rm X}]$ with $f = rg^m$ for some $r \in Q(R) \setminus \{ 0 \}$ and some $m \geq 2$. By \cite[Proposition 3.10]{N16}, $f$ is a closed polynomial. 
\end{proof}

For polynomials $f_1, \ldots, f_n \in R[{\rm X}]$, let $F := (f_1, \ldots, f_n)$. We denote $J(F)$ by the Jacobian matrix of $F$ with respect to variables $x_1, \ldots, x_n$, namely, $J(F) = (\partial f_i/\partial x_j)_{1 \leq i,\: j \leq n}$. The following proposition is a generalization of \cite[Proposition 3.6]{KN15} to the case where the coefficient ring is an integral domain of characteristic zero and $n \geq 1$. 
\begin{prop} \label{prop:2-2} 
Let $R$ be an integral domain of characteristic zero. Let $F := (f_1, \ldots, f_n)$ for polynomials $f_1, \ldots, f_n \in R[{\rm X}]$. If {\rm det}$J(F) \in R \setminus \{ 0 \}$ and $Q(R)[f_i] \cap R[{\rm X}] = R[f_i]$ for $1 \leq i \leq n$, then these polynomials $f_1, \ldots, f_n$ are closed polynomials. In particular, for $g \in R[{\rm X}] \setminus R$ satisfying $Q(R)[g] \cap R[{\rm X}] = R[g]$, if $\hat{g} = {\rm gcd}(g_{x_1}, \ldots, g_{x_n}) \in R \setminus \{ 0 \}$, then it is a closed polynomial.  
\end{prop} 
\begin{proof}
Suppose that ${\rm det} J(F) \in R \setminus \{ 0 \}$, where $F = (f_1, \ldots, f_n)$ for $f_i \in R[{\rm X}] = R[x_1, \ldots, x_n]$. Then there exist $g_{ij}\in Q(R) \otimes_RR[{\rm X}]$ such that 
	\begin{center}
	$\displaystyle \frac{\partial f_i}{\partial x_j} = g_{ij} \hat{f_i}$ 
	\end{center}
for $1 \leq i, j \leq n$, here, we note that $\hat{f_i}$ is the polynomial defined over $Q(R)$. Then we have
\begin{equation*}
	\begin{split}
	{\rm det} J(F) & = \sum_{\sigma \in S_n} {\rm sgn (\sigma)}\frac{\partial f_1}{\partial x_{\sigma (1)}} \cdots \frac{\partial f_n}{\partial x_{\sigma (n)}} \\ 
	& = \sum_{\sigma \in S_n} {\rm sgn (\sigma)}g_{1 \sigma (1)} \hat{f_1} \cdots g_{n \sigma (n)} \hat{f_n} \\ 
	& = (\hat{f_1} \cdots \hat{f_n}) \cdot \sum_{\sigma \in S_n} {\rm sgn (\sigma)}g_{1 \sigma (1)} \cdots g_{n \sigma (n)},       
	\end{split}
\end{equation*} 
where $S_n$ is the symmetric group on $n$ elements. 
Since ${\rm det} J(F) \in R \setminus \{ 0 \}$, $\hat{f_i} \in Q(R) \setminus \{ 0 \}$, so $\deg \hat{f_i} = 0$ for $1 \leq i \leq n$. Therefore $\hat{f_i}$ satisfies the inequality of Theorem \ref{thm:N16} (b) for ${\bf w} = (1, \ldots, 1)$ if $\deg f_i \geq 2$. Otherwise $\deg f_i = 1$. By Theorem \ref{thm:N16}, $f_i$ is a closed polynomial for $1 \leq i \leq n$. 
\end{proof}

\begin{example} \label{ex:2-3}
{\rm 
Let $R := \C[x]$ and $B := R[y, z, t] = \C[x, y, z, t]$ be the polynomial rings over $\C$. For $n \geq 1$, let $v_n := y + x^n(xz + y(yt + z^2)) \in B$. This is often called an \emph{$n$-th V\'{e}n\'{e}reau polynomial}. If $n \geq 2$, then it is known to be a coordinate over $R$, however, we do not know whether $v_1$ is a coordinate over $R$ or not (see \cite[Example 3.18]{EMS136} and \cite[Corollary 14]{Lew13}). Here, we can show that $v_1$ is a closed polynomial over $R$ (of course, $v_n$ is a closed polynomial for $n \geq 2$). 
}
\end{example}
\begin{proof}
Since $c(v_1 - v_1(0, 0, 0)) = {\rm gcd}(1, x^2, x^1) = 1$, by Lemma \ref{lem:1-3}, $Q(R)[v_1] \cap B = R[v_1]$. Furthermore, $\widehat{v_1} = {\rm gcd}((v_1)_y, (v_1)_z, (v_1)_t) = 1$, so this is a closed polynomial over $R$. In other words, $\C[x, v_1]$ is integrally closed in $B = \C[x, y, z, t]$. 
\end{proof}

The following lemma is a generalization of \cite[Proposition 4.1]{N18} to the case where the coefficient ring is an integral domain of characteristic zero and $n \geq 1$. 
\begin{lem} \label{lem:2-4} 
Let $R$ be an integral domain of characteristic zero. For a non-constant polynomial $f \in R[{\rm X}] \setminus R$, the following conditions are equivalent to each other{\rm :} 
	\begin{enumerate}
		\item
		$\deg \hat{f} = \deg f - 1$. 
		\item
		There exist $r_1, \ldots, r_n \in Q(R)$ with $(r_1, \ldots, r_n) \neq (0, \ldots, 0)$ such that $f \in Q(R)[r_1x_1 + \cdots + r_nx_n]$. 
	\end{enumerate} 
\end{lem}
\begin{proof}
{\bf (1) $\Rightarrow$ (2)} Let $d = \deg f$. 
There exist $r_1, \ldots, r_n \in Q(R) \otimes_RR[{\rm X}]$ such that $f_{x_i} = r_i \hat{f}$ for $1 \leq i \leq n$. We may assume that $f_{x_1} \neq 0$. Then 
	\begin{center}
	$d - 1 = \deg \hat{f} \leq \deg f_{x_1} \leq d - 1$,  
	\end{center}
so we have $\deg f_{x_1} = d - 1 = \deg \hat{f}$ and $r_1 \in R \setminus \{ 0 \}$. For $1 \leq i \leq n$ with $f_{x_i} \neq 0$, using the same argument, we have $r_i \in Q(R) \setminus \{ 0 \}$. On the other hand, for $1 \leq i \leq n$ with $f_{x_i} = 0$, we have $r_i = 0$. So $r_i$ is either a non-zero constant polynomial or $0$ for $1 \leq i \leq n$. Set $g := r_1x_1 + \cdots + r_nx_n$. Since $\deg g = 1$, we see easily that $g$ is a closed polynomial in $Q(R) \otimes_RR[{\rm X}]$. By \cite[Theorem 3.1]{KN15}, there exists a $Q(R)$-derivation $D$ on $Q(R) \otimes_RR[{\rm X}]$ such that $\Ker D = Q(R)[g]$. Then 
\begin{equation*}
	\begin{split}
	D(f) & = D(x_1)f_{x_1} + \cdots + D(x_n)f_{x_n} \\ 
	& = D(x_1)r_1\hat{f} + \cdots + D(x_n)r_n\hat{f} \\ 
	& = D(g) \hat{f} \\ 
	& = 0.
	\end{split}
\end{equation*} 
Therefore $f \in \Ker D = Q(R)[g]$. 

{\bf (2) $\Rightarrow$ (1)} 
Let $d = \deg f$ and $g := r_1x_1 + \cdots + r_nx_n$. Since $f \in Q(R)[g]$, there exists $u(t) \in Q(R)[t]$ of degree $d$ with $f = u(g)$. Then $f_{x_i} = r_iu'(g)$ for $1 \leq i \leq n$, where $u'(t) = du(t)/dt$. Then $\deg u'(g) = d - 1$ and $u'(g)$ divides $\hat{f}$. So we have 
	\begin{center}
	$\deg u'(g) \leq \deg \hat{f} \leq d - 1$. 
	\end{center}
Therefore $\deg \hat{f} = d - 1$. 
\end{proof}

By using this lemma, we get the following result. This is also a generalization of \cite[Corollary 4.2]{N18} to the case where the coefficient ring is an integral domain of characteristic zero and $n \geq 1$. 
\begin{thm} \label{thm:2-5} 
Let $R$ be an integral domain of characteristic zero. For a non-constant polynomial $f \in R[{\rm X}] \setminus R$ of prime degree such that $Q(R)[f] \cap R[{\rm X}] = R[f]$, the following conditions are equivalent to each other{\rm :}   
	\begin{enumerate}
		\item
		$f$ is a closed polynomial. 
		\item
		$\deg \hat{f} < \deg f - 1$. 
	\end{enumerate} 
\end{thm}
\begin{proof}
{\bf (1) $\Rightarrow$ (2)} 
Suppose that $\deg \hat{f} = \deg f - 1$. By Lemma \ref{lem:2-4}, there exist $r_1, \ldots, r_n \in Q(R)$ with $(r_1, \ldots, r_n) \neq (0, \ldots, 0)$ satisfying $f \in Q(R)[g]$, where $g := r_1x_1 + \cdots + r_nx_n$. Since $\deg f$ is prime, especially $\deg f \geq 2$, we have $Q(R)[f] \subsetneq Q(R)[g]$. By \cite[Theorem 3.1]{KN15}, $f$ is not a closed polynomial. 

{\bf (2) $\Rightarrow$ (1)} 
Suppose that $\deg \hat{f} < \deg f - 1$. Since $\deg f$ is prime, $\ld_{\bf w}(f) = \deg f \geq 2$, where ${\bf w} = (1, \ldots, 1)$. Then 
	\begin{center}
	$\displaystyle \frac{\ld_{\bf w}(f) - 1}{\ld_{\bf w}(f)} \deg f  = \frac{\deg f - 1}{\deg f} \deg f = \deg f - 1$. 
	\end{center}
Therefore we have
	\begin{center}
	$\displaystyle \deg \hat{f} < \deg f - 1 = \frac{\ld_{\bf w}(f) - 1}{\ld_{\bf w}(f)} \deg f$. 
	\end{center}
By Theorem \ref{thm:N16}, $f$ is a closed polynomial. 

\end{proof}

The following result give a characterization of factorially closed polynomials in the case where the coefficient ring is an algebraically closed field of any characteristic and $n \geq 1$. This is a generalization of \cite[Theorem 2.5 (2)]{N18}.  
\begin{thm} \label{thm:2-6} 
Let $k$ be an algebraically closed field. For a non-constant polynomial $f \in k[{\rm X}] \setminus k$, the following conditions are equivalent to each other{\rm :}   
	\begin{enumerate}
		\item
		$f$ is a factorially closed polynomial. 
		\item
		For any $\lambda \in k$, $f - \lambda$ is irreducible. 
	\end{enumerate} 
\end{thm}
\begin{proof}
{\bf (1) $\Rightarrow$ (2)} 
Suppose that $k[f]$ is a factorially closed in $k[{\rm X}]$. If there exists $\lambda \in k$ such that $f - \lambda$ is reducible, then $f - \lambda = gh$ for some $g, h \in k[{\rm X}] \setminus k$. Then $gh \in k[f - \lambda] = k[f]$, however, since $\deg g$ and $\deg h$ are less than $\deg (f - \lambda)$, $g$ and $h$ do not belong to $k[f]$. This is a contradiction. 

{\bf (2) $\Rightarrow$ (1)} 
Let $g, h \in k[{\rm X}] \setminus \{ 0 \}$ such that $gh \in k[f]$. Since $k$ is an algebraically closed field, there exist $\lambda_1, \ldots, \lambda_s \in k$ and $\varepsilon \in k^*$ such that 
\begin{center}
$\displaystyle gh = \varepsilon \prod_{i = 1}^{s}(f - \lambda_i)$.
\end{center}
By reordering $\lambda_1, \ldots, \lambda_s \in k$ if necessary, we have 
$g = \varepsilon_1 \prod_{i = 1}^{r}(f - \lambda_i)$ and $h = \varepsilon_2 \prod_{j = r + 1}^{s}(f - \lambda_j)$, 
for $\varepsilon_1, \varepsilon_2 \in k^*$. Hence $g, h \in k[f]$, so $k[f]$ is factorially closed in $k[{\rm X}]$.  
\end{proof}

By Theorem \ref{thm:2-6}, we can give examples of factorially closed polynomials. In particular, Example \ref{ex:2-1} gives us examples which are (integrally) closed but not factorially closed polynomials. By using Theorem \ref{thm:2-6}, we get the following examples.   

\begin{example} \label{ex:2-7} 
{\rm 
{\bf (a)} Let $\C[x, y, z]$ be the polynomial rings in tree variables over $\C$. We define the polynomial in $\C[x, y, z]$ by 
	\begin{center}
	  $f := x^nz - y^2 - y$, 
	\end{center}
where $n \geq 2$. Then $f$ is a factorially closed polynomial, but is not a coordinate (see \cite[Proposition (ii)]{ML01}). This is often called a \emph{Danielewski surface}. 

{\bf  (b)} Let $\C[x, y, z, t]$ be the polynomial rings in four variables over $\C$. We define the polynomial in $\C[x, y, z, t]$ by 
	\begin{center}
	  $g := x + x^2y + z^2 + t^3$. 
	\end{center}
Then $g$ is a factorially closed polynomial, but is not a coordinate (see \cite[\S 1]{ML96}). This is often called a \emph{Koras-Russell threefold}. 	
}
\end{example}
\begin{proof} 
{\rm \bf (a)} For $\lambda \in \C$, let $f_{\lambda} := f - \lambda$. We assume that $f_{\lambda} = gh$ for some $g, h \in \C[x, y, z] \setminus \{ 0 \}$. Computing the $z$-degree of $f_{\lambda} = gh$, we may assume that $\deg_zg = 1$ and $\deg_zh = 0$. Here, we write $g = g_1z + g_2$ for $g_1, g_2 \in \C[x, y]$. Then we have $x^n = g_1h$ and $-y^2 - y - \lambda = g_2h$. Hence $\deg_yh = \deg_xh = 0$, which means $h \in \C^*$. Therefore $f_{\lambda}$ is irreducible for any $\lambda \in \C$. By Theorem \ref{thm:2-6}, $f$ is a factorially closed polynomial. 

{\rm \bf (b)} For $\lambda \in \C$, let $g_{\lambda} := g - \lambda$. We assume that $g_{\lambda} = pq$ for some $p, q \in \C[x, y, z, t] \setminus \{ 0 \}$. Computing the $y$-degree of $g_{\lambda} = pq$, we may assume that $\deg_yp = 1$ and $\deg_yq = 0$. Here, we write $p = p_1y + p_2$ for $p_1, p_2 \in \C[x, z, t]$. Then we have $x^2 = p_1q$ and $x + z^2 + t^3 - \lambda = p_2q$. By the first equation, we have $\deg_zq = \deg_tq = 0$ and $q$ is a component of $x^2$. If $x$ divides $q$, then this contradicts the second equation. Thus $\deg_xq = 0$, so $q \in \C^*$. By Theorem \ref{thm:2-6}, $g$ is a factorially closed polynomial. 
\end{proof}

As the end of this section, we show a relation between factorially closed polynomials and Darboux polynomials. Suppose that $n = 2$. For $f \in R[{\bf X}] = R[x, y]$, we define an $R$-derivation $\Delta_f$ by 
	\begin{center}
	  $\displaystyle \Delta_f := -f_y\frac{\partial}{\partial x} + f_y \frac{\partial}{\partial y}$. 
	\end{center} 
\begin{prop} \label{prop:2-8} 
Let $k$ be an algebraically closed field of characteristic zero and let $f \in k[x, y] \setminus k$ be a non-constant polynomial. If $f$ is a factorially closed polynomial, then $\Delta_f$ has no Darboux polynomials any other than elements of the kernel of $\Delta_f$.  
\end{prop}
\begin{proof} 
We define a morphism $\Phi_f : \Spec k[x, y] \to \Spec k[f]$ by the inclusion $k[f] \subset k[x, y]$. By Proposition \ref{thm:2-6}, every fiber of $\Phi_f$ is irreducible and reduced, in particular it is a fibration. By \cite[Corollary 2.4]{Dai97}, gcd($f_x, f_y$) $= 1$, so $\Delta_f$ is irreducible. Moreover $k(x, y)^{\Delta_f}$ contains $k(f)$. Therefore $f$ and $\Delta_f$ satisfy the assumptions of \cite[Lemma 2.4]{Miy95}. By \cite[Lemma 2.4 (2)]{Miy95}, $X_{\Delta_f} = 0$, which means that if $g$ is a Darboux polynomial of $D$, then $g \in k[x, y]^D$.  
\end{proof}

\section{The kernel of a monomial derivation on $R[x, y]$}
Let $R$ be an integral domain containing $\Q$. In this section, we study the kernels of $R$-derivations on the polynomial ring $R[x,y]$ in two variables $x$ and $y$ over $R$. Let $D$ be an $R$-derivation on $R[x, y]$. We denote the divergence of $D$ by $\dg(D)$, namely, $\dg(D) := \partial D(x)/\partial x + \partial D(y)/\partial y$. A non-zero $R$-derivation $D$ on $R[x,y]$ is said to be {\it monomial} if $D(x)$ and $D(y)$ are monomials, here we assume that a monomial may not be monic. By using results on closed polynomials in the previous section, we determine generators of the kernel of monomial derivations on $R[x, y]$. 

\begin{lem} \label{lem:3-1} 
Let $R$ be an integral domain containing $\Q$ and let $D$ be an $R$-derivation on $R[x, y]$. If $\dg(D) = 0$, then there exists $f \in R[x, y] \setminus R$ such that $D(f) = 0$. 
\end{lem}
\begin{proof}
Let $p := D(x)$ and $q := D(y)$. Then we can write $p, q$ as below: 
	\begin{center}
	  $\displaystyle p = \sum_{m, n \geq 0}[m, n]_px^my^n$ and $\displaystyle q = \sum_{m, n \geq 0}[m, n]_qx^my^n$, 
	\end{center}
where $[m, n]_p, [m, n]_q \in R$. Then 
	\begin{enumerate}
	  \item[]
	  $\displaystyle \frac{\partial p}{\partial x} = \sum_{m \geq 1, n \geq 0} m[m, n]_px^{m - 1}y^n = \sum_{m, n \geq 1} m[m, n - 1]_px^{m - 1}y^{n - 1}$, 
	  \item[]
	  $\displaystyle \frac{\partial q}{\partial y} = \sum_{m \geq 0, n \geq 1} n[m, n]_qx^my^{n - 1} = \sum_{m, n \geq 1} n[m - 1, n]_qx^{m - 1}y^{n - 1}$. 
	\end{enumerate}
Since $0 = \dg(D) = \partial p/\partial x + \partial q/\partial y$, we have $-n^{-1}[m, n - 1]_p = m^{-1}[m - 1, n]_q$ for $m, n \geq 1$. Here, we define a polynomial $f \in R[x, y]$ by $f := \sum_{m, n \geq 0}[m, n]_fx^my^n$, where $[0, 0]_f := 0$, $[m + 1, 0]_f := (m - 1)^{-1}[m, 0]_q$ for $m \geq 0$ and $[m, n + 1]_f := -(n + 1)^{-1}[m, n]_p$ for $m, n \geq 0$. 
Then 
	\begin{center}
	  $\displaystyle f = \sum_{m \geq 0} \frac{1}{m + 1}[m, 0]_qx^{m + 1} + \sum_{m, n \geq 0} - \frac{1}{n + 1}[m, n]_px^my^{n + 1}$. 
	\end{center}
Thus $f_x = q$, $f_y = - p$, so $D(f) = f_xp + f_yq = 0$. 
\end{proof} 
In the case where $R$ is a UFD, the kernel of a non-zero derivation on $R[x, y]$ is generated by one polynomial (see \cite[Corollary 3.2]{EK04}) and it is integrally closed in $R[x, y]$. Thus, if $R[x,y]^D \neq R$, then it is generated by a closed polynomial. So, to determine a generator of the kernel of a derivation on $R[x, y]$, it is sufficient  that we find a closed polynomial which is vanished by the derivation. Indeed, the following holds  true.
\begin{lem} \label{lem:3-2}
Let $R$ be a UFD of characteristic zero and let $D$ be a non-zero $R$-derivation on $R[x,y]$. If there exist a closed polynomial $f \in R[x, y] \setminus R$ such that $D(f) = 0$ and $Q(R)[f] \cap R[x, y] = R[f]$, then $R[x, y]^D = R[f]$.  
\end{lem}
\begin{proof}
Suppose that $R[x, y]^D = R[g]$ and $D(f) = 0$, where $f$ is a closed polynomial satisfying $Q(R)[f] \cap R[x, y] = R[f]$. Here, we note that $Q(R)[f]$ is also integrally closed in $Q(R)[x, y]$. Since $D(f) = 0$, we have $R[f] \subset R[x, y]^D = R[g]$, so we can write $f$ as a polynomial in $g$ by below. 
	\begin{center}
	  $f = u_0g^m + u_1g^{m - 1} + \cdots + u_{m - 1}g + u_m$, 
	\end{center}
where $u_i \in R$ and $u_0 \neq 0$. By multiplying $u_0^{-1}$ on the both sides, $g$ is integral over $Q(R)[f]$, so $g \in Q(R)[f]$. Therefore $Q(R)[f] = Q(R)[g]$, hence $g = u_0^{- 1}f$. This means that $g \in Q(R)[f] \cap R[x, y] = R[f]$, so $R[f] = R[g] = R[x, y]^D$.   
\end{proof}
   

The following is the main result in this section, which gives the classification of kernels of monomial derivations on $R[x, y]$, where $R$ is a UFD containing $\Q$. For the following discussions, we denote $\partial/\partial x$ (resp. $\partial/\partial y$) by $\partial_x$ (resp. $\partial_y$).
\begin{thm} \label{thm:3-3}
Let $D$ be a non-zero $R$-derivation on the polynomial ring $R[x,y]$ in two variables over a UFD $R$ containing $\Q$. Assume that $D(x)$ and $D(y)$ are monomial, ${\rm gcd}(D(x), D(y)) = 1$ and $D$ is none of the following {\rm (1)}--{\rm (3)}{\rm :} 
	\begin{enumerate}
	  \item
	  $\partial_x$ or $\partial_y$, 
	  \item
	  $ay^{m}\partial_x + bx^{n} \partial_y$, where $m, n \in \Z_{\geq 0}$ and $a, b \in R \setminus \{ 0 \}$, 
	  \item
	  $nx \partial_x - my\partial_y$, where $m$ and $n$ are positive integers. 
	\end{enumerate}
Then $R[x,y]^D = R$.
\end{thm}


To prove Theorem \ref{thm:3-3} we show the following two lemmas. 
First of all, by the following lemma, we see that for derivations as in Theorem \ref{thm:3-3} {\rm (1)--(3)}, their kernels are generated by a closed polynomial, that is, they contained non-constant polynomials as kernels. 
\begin{lem} \label{lem:3-4}
For the derivations as in Theorem \ref{thm:3-3} {\rm (1)}, {\rm (2)} and {\rm (3)}, the following assertions holds true. 
	\begin{enumerate}
	  \item[{\rm (a)}]
	  $D_1 := \partial_x$. Then $R[x, y]^{D_1} = R[y]$.
	  \item[{\rm (b)}]
	  $D_2 := ay^{m}\partial_x + bx^{n} \partial_y$, where $m, n \in \Z_{\geq 0}$ and $a, b \in R \setminus \{ 0 \}$ with ${\rm gcd}(a, b) = 1$. Then $R[x, y]^{D_2} = R[b(m+1)x^{n+1} - a(n+1)y^{m+1}]$.  
	  \item[{\rm (c)}]
	  $D_3 := nx \partial_x - my\partial_y$, where $m$ and $n$ are relatively prime positive integers. Then $R[x, y]^{D_3} = R[x^my^n]$. 
	\end{enumerate}
\end{lem}
\begin{proof}
{\rm \bf (a)} Obvious. 

{\rm \bf (b)}
Since $\dg (D_2) = 0$, by Lemma \ref{lem:3-1}, there exists $f \in R[x, y] \setminus R$ such that $D_2(f) = 0$. By the proof of Lemma \ref{lem:3-1}, we can write $f$ as 
	\begin{center}
	  $\displaystyle f = \frac{1}{n + 1}bx^{n + 1} - \frac{1}{m + 1}ay^{m + 1}$. 
	\end{center}
By Lemma \ref{lem:1-3}, we see that $Q(R)[f] \cap R[x, y] = R[f]$. Moreover, we can check easily that ${\rm gcd}(f_x, f_y) = 1$. By Proposition \ref{prop:2-2}, $f$ is a closed polynomial, also $(m + 1)(n + 1)f$ is a closed polynomial. Thus $R[x, y]^{D_2} = R[(m + 1)(n + 1)f]$.  

{\rm \bf (c)} 
Let $g = x^{m -1}y^{n - 1} \in R[x, y]$. Then $\dg (gD_3) = 0$. By Lemma \ref{lem:3-1}, we can construct a polynomial $h \in R[x, y]$ by $h = - x^my^n$. Then $gD_3(h) = 0$ and $Q(R)[h] \cap R[x, y] = R[h]$. Since $m$ and $n$ are relatively prime, by Example \ref{ex:2-1}, $h$ is a closed polynomial. Thus $R[x, y]^{D_3} = R[x, y]^{gD_3} = R[h] = R[x^my^n]$. 
\end{proof}

Next, we show the following lemma. This gives some types of derivations whose kernel has only constant polynomials. 
\begin{lem} \label{lem:3-5}
For $g \in R[x, y] \setminus \{ 0 \}$, let $D = \partial_x + g\partial_y$. If $\deg_y g \geq 1$, then $R[x,y]^D = R$.
\end{lem}
\begin{proof} 
Let $g = b_0y^{\deg_yg} +$ (the lower $y$-degree terms), for $b_0 \in R[x] \setminus \{ 0 \}$. 
We take any element $h \in R[x, y] \setminus \{ 0 \}$ and put
	\begin{center}
	  $h = a_0 y^s + a_1 y^s + \cdots + a_{s-1} y + a_s$,
	\end{center}
where $s = \deg_y h (\geq 0)$, $a_0, \ldots, a_s \in R[x]$ and $a_0 \neq 0$. Then
	\begin{center}
	  $D(h) = (D(a_0)y^s + \cdots + D(a_s)) + g(sa_0y^{s-1} + \cdots + a_{s-1})$
	\end{center}
Since $\deg_y g \geq 1$, we have $s \leq s - 1 + \deg_y g$. 

Now, we suppose that $D(h) = 0$. If $s < s - 1 + \deg_yg$, then by comparing the coefficients of $y^s$ in the equation $D(h) = 0$, we obtain the equality $a_0b_0s = 0$, so $s = 0$. Then $0 = D(h) = D(a_0) = \partial_x(a_0)$, hence $h = a_0 \in R$. 
On the other hand, if $s = s - 1 + \deg_yg$, then we obtain the equality $D(a_0) + a_0b_0s = 0$. Since $\deg_x D(a_0) < \deg_x a_0 \leq \deg_xa_0b_0$, we have $s = 0$. Hence $h = a_0 \in R$. 
\end{proof}

Now, we shall prove Theorem \ref{thm:3-3}. 
\begin{proof}[Proof of Theorem \ref{thm:3-3}]
From now on, we assume that $D$ is none of (1)--(3) of Theorem \ref{thm:3-3} and prove that $R[x,y]^D = R$. Let $K := Q(R)$. We denote $D_K$ by the $K$-derivation on $K[x, y]$ which is the extension of $D$. To prove $R[x,y]^D = R$, it is sufficient to show that $K[x, y]^{D_K} = K$. Therefore we enough to show that for the following $K$-derivation $D$, the kernel of that is equal to $K$: 
	\begin{center} 
	  $D = x^m \partial_x + ay^n \partial_y$,
	\end{center}
where $a \in k^*$, $m,n \in \Z_{\geq 0}$. If $m = 0$ and $n \geq 1$, then $D$ is the form in Lemma \ref{lem:3-5}. So we already know that the kernel of it is $K$. Therefore we may assume that $n \geq m \geq 1$. Let $d$ be the greatest common divisor of $m - 1$ and $n - 1$ as integers,  $m' := (m - 1)/d$ and $n' := (n - 1)/d$,  here we assume $m' = n' =1$ if $m = n = 1$. We set ${\bf w} := (n', m')$ and consider the ${\bf w}$-grading on $K[x, y]$. Then we can easily check that if $f \in K[x, y]$ is ${\bf w}$-homogeneous then so is $D(f)$. 

Let $f$ be any non-zero element of $K[x, y]^D$. In order to prove $K[x, y]^D = K$, we may assume that $f$ is ${\bf w}$-homogeneous. 
Then we have $(\alpha_0, \beta_0) \in (\Z_{\geq 0})^2$ such that
	\begin{center}
	  $f = \displaystyle \sum_{\substack{i\ge 0\\\beta_{0}-in'\ge 0}}c_{i}x^{\alpha_{0}+im'}y^{\beta_{0}-in'}, \quad \quad \quad (*)$
	\end{center}
where $c_i \in K$. Since $f \in K[x,y]^D$, we have
	\begin{align*}
	  0 & = D(f)\\
	  & = \sum_{\substack{i\ge 0\\\beta_{0}-in'\ge 0}}c_{i}\left((\alpha_{0} + im')x^{\alpha_i}y^{\beta_{0} - in'} + (\beta_{0} - in')ax^{\alpha_{0} + im'}y^{\beta_i} \right), \quad \quad \quad (**)
	\end{align*}
where $\alpha_i = \alpha_{0} + im' + m - 1$ and $\beta_i = \beta_{0} - jn' + n -1$. 
Here we set the following subsets $A$ and $B$ of $(\Z_{\geq 0})^2$:
	\begin{center}
	  $A := \{(\alpha_{0} + im' + m - 1, \beta_{0} - in')\mid i\ge 0, \beta_{0}-in'\ge 0\}$,
	\end{center}
	\begin{center}
	  $B := \{(\alpha_{0} + jm', \beta_{0} - jn' + n -1)\mid j\ge 0, \beta_{0}-jn'\ge 0\}$.
	\end{center}
Suppose that $A \cap B = \emptyset$. Then, by taking $i=0$ in $A$, we see from $(**)$ that $c_0 \alpha_0 = 0$. So, $\alpha_0 = 0$. Similarly, we have $\beta_0 = 0$. Hence $f = c_0 x^{\alpha_0} = c_0 \in K$.

Suppose that $A \cap B \not= \emptyset$. Then there exist $i, j \in \Z_{\geq 0}$ such that 
	\begin{center}
		$
		  \begin{array}{l}
		    \alpha_{0} + im' + m - 1 = \alpha_{0} + jm', \\
		    \beta_{0} - in' = \beta_{0} - jn' + n - 1, \\
		    \beta_0- in' \geq 0, \\
		    \beta_0- jn' \geq 0.
		  \end{array}
		$
	\end{center}
Then $(j - i)m' = m - 1$ and $(j - i) n' = n - 1$. Here we may assume that $j \geq i$. Then $j - i = d$. We consider the cases $n \geq 2$ and $n = 1$ separately.

\vspace{1mm}
\noindent
Case: $n \geq 2$. Then $j > i$. By considering the term $i = 0$ in $(**)$, we have $c_0 \beta_0 a = 0$. So $\beta_0 = 0$. 
Since $f = c_0 x^{\alpha_0} \in K[x, y]^D$, we have $\alpha_0 = 0$. Therefore, $f = c_0 \in K$. 

\vspace{1mm}
\noindent
Case: $n = 1$. Then $m' = n' = 1$ and so $i = j$. By $(**)$, we have
	\begin{center}
	  $
	    \begin{array}{l}
	      c_{0}\alpha_{0} + c_{0}\beta_{0}a = 0, \\
	      c_{1}(\alpha_{0}+1) + c_{1}(\beta_{0}-1)a = 0, \\
	      ~~~\vdots\\
	      c_{\beta_{0}-1}(\alpha_{0}+\beta_{0}-1) + c_{\beta_{0}-1}a = 0, \\
	      c_{\beta_{0}}(\alpha_{0}+\beta_{0}) + c_{\beta_{0}}a = 0.
	    \end{array}
	    $
	\end{center}
Since $c_0 \not= 0$, we have $\alpha_0 + \beta_0 a = 0$. If $\beta_0 > 0$, then $a = - \alpha_{0}/\beta_{0} \in \Q_{<0}$. So $D$ is (3) of Theorem \ref{thm:3-3}. If $\beta_0 =0$, then $\alpha_0 = 0$ and hence $f \in K$. 
\end{proof} 

We note here that the condition ``$R$ is a UFD" is necessary. Even if $D$ is a monomial derivation, the kernel may not be finitely generated over $R$ in the case where $R$ is not a UFD.  We give an example below: 
\begin{example} \label{ex:3-6} 
{\rm (cf.\ \cite[Example 4.4]{EMS136})} 
{\rm 
Let $k$ be a field of characteristic zero and let $k[t]$ be the polynomial ring in one variable. Let $R = k[t^2, t^3]$. Here, we define an $R$-derivation $D$ on $R[x, y]$ by 
	\begin{center}
	  $D = \displaystyle t^2\frac{\partial}{\partial x} + t^3\frac{\partial}{\partial y}$. 
	\end{center}
Then $D$ is a monomial derivation, but $R[x, y]^D = R[f^mt^2 \ | \ m \geq 1]$, where $f = tx - y \notin R[x, y]$. Therefore the kernel of this derivations is not generated by one polynomial, in particular, it needs infinite generators.  
}
\end{example}
\section{The kernel of a monomial derivation on $k(x,y)$}
Let $k[x,y]$ be the polynomial ring in two variables over a field $k$ of characteristic zero and $k(x,y)$ its quotient field.  Recall that for a $k$-derivation $D$ on $k[x,y]$, we denote the same notation $D$ by the $k$-derivation on $k(x, y)$ which is the natural extension of the original $D$, and its kernel is denoted by $k({\rm X})^D$. In this section, by using the argument in \cite[\S 5]{NZ06} and Theorem \ref{thm:3-3}, we determine the non-zero monomial derivations $D$ on $k[x,y]$ such that $Q(k[x,y]^D) \not= k(x,y)^D$. 

Let $D$ be a monomial $k$-derivation on $k[x,y]$. In order to study $k(x,y)^D$, by switching the role of $x$ and $y$, we may assume that the following conditions are satisfied:
\begin{itemize}
\item[(i)]
$D(x)$ is monic.
\item[(ii)]
$\operatorname{gcd}(D(x), D(y)) = 1$.
\item[(iii)]
$\deg D(x) \leq \deg D(y)$ provided $D(y) \not= 0$.
\end{itemize} 
For the following discussions, we denote also $\partial/\partial x$ (resp. $\partial/\partial y$) by $\partial_x$ (resp. $\partial_y$). The following is the main result in this section. 
\begin{thm} \label{thm:4-1}
Let $D$ be a non-zero monomial $k$-derivation on the polynomial ring $k[x,y]$ in two variables over a field $k$ of characteristic zero. Assume that $D$ satisfies the above three conditions {\rm (i)--(iii)}, $k[x,y]^D = k$ and $k(x,y)^D \not= k$. Then $D$ is one of the following {\rm (1)--(3)}.
	\begin{itemize}
	  \item[(1)]
	  $D = \partial_x + ax^m y^{n+1} \partial_y$, where $m \in \Z_{\geq 0}$, $n \in \Z_{>0}$ and $a \in k^*$. 
	  \item[(2)]
	  $D = x^{m+1} \partial_x + ay^{n+1} \partial_y$, where $m, n \in \Z_{> 0}$ with $m \leq n$ and $a \in k^*$. 
	  \item[(3)]
	  $D = x \partial_x + ay \partial_y$, where $a$ is a positive rational number. 
	\end{itemize}
\end{thm} 
\noindent 
Let $D$ be a $k$-derivation on $k[x,y]$. If $k[x,y]^D \not= k$, then $Q(k[x,y]^D) = k(x,y)^D$. See \cite[Theorem]{Z00}, which is generalized in \cite{AR02} and \cite{K11}. So Theorem \ref{thm:4-1} also gives the classification of the monomial $k$-derivations $D$ on $k[x,y]$ such that $Q(k[x,y]^D) \not= k(x,y)^D$. 

To prove Theorem \ref{thm:4-1} we show the following two lemmas. 
\begin{lem} \label{lem:4-2}
Let $D = \partial_x + ax^{m} y^{n+1}\partial_y$, where $m, n \in \Z_{\geq 0}$ and $a \in k^*$. Then $k(x, y)^D = k$ if and only if $n = 0$. 
\end{lem}
\begin{proof}
If $n \geq 1$, then $n x^{m+1} + (m+1)a^{-1} y^{-n} \in k(x, y)^D \setminus k$. We assume that $n = 0$. By Lemma \ref{lem:3-5}, $k[x, y]^D = k$. Let $f \in k[x, y] \setminus k$ be a non-constant polynomial and put
	\begin{center}
	  $f = a_s y^s + a_{s-1} y^s + \cdots + a_1 y + a_0$,
	\end{center}
where $s = \deg_y f (\geq 0)$, $a_0, \ldots, a_s \in k[x]$ and $a_s \not= 0$. Assume that $g:= D(f)/f \in k[x, y]$, namely, $f$ is a Darboux polynomial of $D$. 

Assume further that $a_0 \not= 0$, i.e., $y \not| f$. Since $f$ is non-constant and $D(f) = gf$, $g \not= 0$. We have
	\begin{center}
	  $s \geq \deg_y D(f) = \deg_y g + \deg_y f = \deg_y g +s$.
	\end{center}  
This implies $g \in k[x]$. Comparing the constant terms with respect to $y$ in the equation $D(f) = gf$, we have $a'_0 = ga_0$, where $a'_0$ is the derivative of $a_0$ with respect to $x$, which is a contradiction. Hence $a_0 = 0$.

The argument in the previous paragraph implies that $f$ can be expressed as $f = f_1 y^t$, where $t \in \Z_{> 0}$, $f_1 \in k[x, y]$ and $y \not| f_1$. By \cite[Proposition 2.4]{NZ06}, $f_1$ is also a Darboux polynomial of $D$ and so $f_1 \in k^*$. Therefore, $f$ can be expressed as $f = a_sy^s$, where $a_s \in k^*$. We infer from \cite[Proposition 2.5]{NZ06} that $k(x, y)^D = k$. 
\end{proof}

\begin{lem} \label{lem:4-3} 
Assume that $D = x^{m+1} \partial_x + ay^{n+1} \partial_y$, where $a \in k^*$, $m,n \in \Z_{\geq 0}$ and $m \leq n$, and that $k(x, y)^D \not= k$. Then one of the following conditions {\rm (1)} and {\rm (2)} holds true.
	\begin{itemize}
	  \item[(1)]
	  $m, n > 0$.
	  \item[(2)]
	  $m = n = 0$ and $a \in \Q \setminus \{ 0 \}$. 
	\end{itemize}
\end{lem}
\begin{proof}
If $D$ satisfies the condition (1) (resp.\ (2)), then $ma^{-1}y^{-n} - nx^{-m} \in k(x, y)^D \setminus k$ (resp.\ $x^p y^{-q}  \in k(x, y)^D \setminus k$, where $p$ and $q$ are relatively prime integers such that $a = p/q$). We consider the following cases separately.

\vspace{1mm}
\noindent
Case: $m = n = 0$ and $a \not\in \Q$. By Theorem \ref{thm:3-3}, $B^D = k$. Let $f \in R \setminus k$ be a non-constant polynomial and put
	\begin{center}
	  $f = a_s y^s + a_{s-1} y^s + \cdots + a_1 y + a_0$,
	\end{center}
where $s = \deg_y f (\geq 0)$, $a_0, \ldots, a_s \in k[x]$ and $a_s \not= 0$. Assume that $f$ is a Darboux polynomial and set $g = D(f)/f$. 

Assume further that $a_0 \not= 0$, i.e., $y \not| f$. 
Since $f$ is non-constant and $D(f) = gf$, $g \not= 0$. We have $\deg_y D(f) = \deg_y g + s$. Since $\deg_y D(f) \leq s$, $g \in k[x]$ and $xa'_0 = g a_0$. So, $n_0:= g = \deg_x a_0 \in \Z_{>0}$ and $a_0 = b x^{n_0}$ for some $b \in k^*$. 
Assume further that $s > 0$. Comparing the highest terms with respect to $y$ in the equation $D(f) = gf$, we have $x a'_s = (n_0 -sa) a_s$. Then $n_0 - sa = \deg_x a_s$ and so $a \in \Q$. This is a contradiction. Therefore, $s = 0$, i.e., $f = b x^{n_0}$. 

Assume next that $a_0 = 0$. We set as $f = f_1 y^t$, where $t \in \Z_{> 0}$, $f_1 \in B$ and $y \not| f_1$. Then $f_1$ is also a Darboux polynomial of $D$. So the argument in the previous paragraph implies that $f_1 = b x^{\deg_x f_1}$ for some $b \in k^*$. 

Therefore, $f$ can be expressed as $b x^i y^j$ for some $i, j \in \Z_{\geq 0}$ and $b \in k^*$. Since $a \not\in \Q$, we infer from \cite[Proposition 2.5]{NZ06} that $k(x, y)^D = k$. 

\vspace{1mm}
\noindent
Case: $n = 0$, $m \geq 1$. Set $D_1 = \partial_x + ax^{m+1}y \partial_y$, where $m$ and $a$ are the same as in $D$. By Lemma \ref{lem:4-2}, $k(x, y)^{D_1} = k$. Let $\sigma: k(x, y) \to k(x, y)$ be the $k$-automorphism defined by $\sigma(x) = x^{-1}$ and $\sigma(y) = y{-1}$. Then $D = -x^{m-1} \sigma D_1 \sigma^{-1}$. Hence $k(x, y)^D = k$.

\vspace{1mm}
\noindent
Case: $m = 0$, $n \geq 1$. By using the same argument as in the previous case, we have $k(x, y)^D = k$. 
The proof of Lemma \ref{lem:4-3} is thus verified. 
\end{proof}
Theorem \ref{thm:4-1} is a consequence of Theorem \ref{thm:3-3}, Lemmas \ref{lem:4-2} and \ref{lem:4-3}.


\end{document}